\pdfoutput=1
\RequirePackage{ifpdf}
\ifpdf 
\documentclass[pdftex]{sigma}
\else
\documentclass{sigma}
\fi

\usepackage{mathtools}
\usepackage[mathscr]{euscript}
\usepackage{xcolor,colortbl}

\numberwithin{equation}{section}

\newtheorem{Theorem}{Theorem}[section]
\newtheorem{Corollary}[Theorem]{Corollary}
\newtheorem{Lemma}[Theorem]{Lemma}
\newtheorem{Proposition}[Theorem]{Proposition}
 { \theoremstyle{definition}
\newtheorem{Definition}[Theorem]{Definition}

\newtheorem{Example}[Theorem]{Example}
\newtheorem{Remark}[Theorem]{Remark} }

\DeclareMathOperator{\rank}{rank}

\begin{document}

\allowdisplaybreaks

\newcommand{\arXivNumber}{1612.09439}

\renewcommand{\thefootnote}{}

\renewcommand{\PaperNumber}{045}

\FirstPageHeading

\ShortArticleName{Hodge Numbers from Picard--Fuchs Equations}

\ArticleName{Hodge Numbers from Picard--Fuchs Equations\footnote{This paper is a~contribution to the Special Issue on Modular Forms and String Theory in honor of Noriko Yui. The full collection is available at \href{http://www.emis.de/journals/SIGMA/modular-forms.html}{http://www.emis.de/journals/SIGMA/modular-forms.html}}}

\Author{Charles F.~DORAN~$^{\dag^1}$, Andrew HARDER~$^{\dag^2}$ and Alan THOMPSON~$^{\dag^3\dag^4}$}

\AuthorNameForHeading{C.F.~Doran, A.~Harder and A.~Thompson}

\Address{$^{\dag^1}$~Department of Mathematical and Statistical Sciences, 632 CAB, University of Alberta, \\
\hphantom{$^{\dag^1}$}~Edmonton, AB, T6G 2G1, Canada}
\EmailDD{\href{mailto:charles.doran@ualberta.ca}{charles.doran@ualberta.ca}}

\Address{$^{\dag^2}$~Department of Mathematics, University of Miami, 1365 Memorial Drive, Ungar 515,\\
\hphantom{$^{\dag^2}$}~Coral Gables, FL, 33146, USA}
\EmailDD{\href{mailto:a.harder@math.miami.edu}{a.harder@math.miami.edu}}

\Address{$^{\dag^3}$~Mathematics Institute, Zeeman Building, University of Warwick, Coventry, CV4 7AL, UK}
\EmailDD{\href{mailto:a.thompson.8@warwick.ac.uk}{a.thompson.8@warwick.ac.uk}}

\Address{$^{\dag^4}$~DPMMS, Centre for Mathematical Sciences, University of Cambridge, Wilberforce Road,\\
\hphantom{$^{\dag^4}$}~Cambridge, CB3 0WB, UK}

\ArticleDates{Received January 20, 2017, in f\/inal form June 12, 2017; Published online June 18, 2017}

\Abstract{Given a variation of Hodge structure over $\mathbb{P}^1$ with Hodge numbers $(1,1,\dots,1)$, we show how to compute the degrees of the Deligne extension of its Hodge bundles, following Eskin--Kontsevich--M\"oller--Zorich, by using the local exponents of the corresponding Picard--Fuchs equation. This allows us to compute the Hodge numbers of Zucker's Hodge structure on the corresponding parabolic cohomology groups. We also apply this to families of elliptic curves, K3 surfaces and Calabi--Yau threefolds.}


\Classification{14D07; 14D05; 14J32}

\renewcommand{\thefootnote}{\arabic{footnote}}
\setcounter{footnote}{0}

\section{Introduction}
The goal of this paper is to compute the Hodge numbers of the parabolic cohomology groups of a variation of Hodge structure (VHS) in the case where we know the Picard--Fuchs equation.

In more detail, let us assume that $C$ is a smooth quasiprojective curve which bears an $\mathbb{R}$-VHS of weight $k$. If $\mathbb{V}$ is the underlying local system and $j\colon C \hookrightarrow \overline{C}$ is the embedding of $C$ into its smooth completion, then Zucker \cite{htdcl2cpm} showed that $H^1(\overline{C}, j_* \mathbb{V})$ (which we will call the \emph{parabolic cohomology} of $\mathbb{V}$) bears a pure Hodge structure of weight $k+1$. Moreover, Zucker also showed that this Hodge structure is closely related to the f\/iltrands of the Leray spectral sequence when our VHS comes from a family of manifolds, so one can use the Hodge structure on the parabolic cohomology groups to understand the Hodge structure of f\/ibrations and vice versa.

More recently, following work of Morrison and Walcher \cite{dbnf, ehalaots}, interest has arisen in determining when the Hodge structure on the parabolic cohomology of specif\/ic families of Calabi--Yau threefolds admits rational $(2,2)$ classes, since such classes correspond to possible normal functions. Morrison and Walcher \cite{dbnf} gave an exciting interpretation of such normal functions of families of Calabi--Yau threefolds in terms of $D$-branes.

Following this, del Angel, M\"uller-Stach, van Straten, and Zuo \cite{hca1pfcy3f} sought further examples of such normal functions, by studying Hodge structures on the parabolic cohomology of families of Calabi--Yau threefolds whose underlying VHS's are pull-backs of VHS's on the thrice-punctured sphere with $b_3 = 4$; such VHS's with $b_3 =4$ were classif\/ied by Doran and Morgan \cite{doranmorgan}, who found $14$ cases. They reduced these computations to the computation of the degree of the quasi-canonical extension of the Hodge bundles of these VHS's. However, it appears as if the authors of \cite{hca1pfcy3f} were not able to compute these degrees in great generality, thus many entries in their tables are left blank. Further progress was made by Holborn and M\"uller-Stach \cite{hnccytls}, however it appears that they were still not able to complete the table from \cite{hca1pfcy3f}. Some related computations of degrees of quasi-canonical extensions of Hodge bundles have also been carried out by Green, Grif\/f\/iths, and Kerr \cite{egpvhs}.

Recently, a preprint of Eskin, Kontsevich, M\"oller, and Zorich \cite[Section 6]{lblefbc} suggested a way to compute the degrees of Hodge bundles of $\mathbb{R}$-VHS's of $(1,1,\dots,1)$-type, when the Picard--Fuchs operators controlling these $\mathbb{R}$-VHS's are known. The goal of this paper is to present the technique of \cite{lblefbc} in greater generality\footnote{As stated, the results of \cite{lblefbc} apply only to VHS's of $(1,1,1,1)$-type, though it is clear, even in their exposition, that their approach is more general.} and apply it to an array of situations. In particular, we of\/fer a completion of the tables of \cite{hca1pfcy3f}.

The application most germane to the previous work of the authors is to the study of threefolds f\/ibred by lattice polarized K3 surfaces. In \cite{flpk3sm,cytfksapec,cytfmqk3s,cytfhrlpk3s}, we constructed Calabi--Yau threefolds by f\/irst choosing convenient families of lattice polarized K3 surfaces $\mathcal{X} \to \mathbb{P}^1$, then performing base change along suitable maps $g\colon \mathbb{P}^1 \rightarrow \mathbb{P}^1$. One of the main results of this work \cite{cytfhrlpk3s} is a classif\/ication of all Calabi--Yau threefolds f\/ibred by $M_n$-polarized K3 surfaces, where $M_n$ denotes the rank $19$ lattice $M_n := E_8 \oplus E_8 \oplus H \oplus \langle -2n \rangle$. In order to restrict the number of cases that we needed to check geometrically, the results and techniques of this paper helped immensely.

In more generality, our results allow us to compute the Hodge numbers of the parabolic cohomology $H^1(\mathbb{P}^1, j_*\mathbb{V})$, where $\mathcal{X} \to \mathbb{P}^1$ is any smooth projective threefold f\/ibred by K3 surfaces over $\mathbb{P}^1$ with generically Picard rank $19$ f\/ibres, and $\mathbb{V}$ is the corresponding local system of cohomology over the complement of the critical values of the f\/ibration. By the Leray spectral sequence \cite[Section 15]{htdcl2cpm}, this is a direct summand of the Hodge structure on $H^3(\mathcal{X},\mathbb{Q})$, and its complement has no $(3,0)$ or $(0,3)$ part, so in fact this computes $h^{3,0}(\mathcal{X})$ and places a lower bound on $h^{2,1}(\mathcal{X})$.

Finally, we note that in the case of an elliptic f\/ibration, the degree of the Hodge bundle can be computed geometrically by putting the f\/ibration into Weierstrass normal form
\begin{gather*}Y^2 = X^3 + g_2X + g_3\end{gather*}
and computing the degrees of $g_2$ and $g_3$. In the case of K3 surface f\/ibrations (and to an even greater extent for Calabi--Yau threefold f\/ibrations) we do not have such a normal form available, because there are many families of K3 surfaces polarized by rank $19$ lattices over $\mathbb{P}^1$. One can construct normal forms corresponding to each polarizing lattice, as we have done in \cite{cytfmqk3s, cytfhrlpk3s} for some specif\/ically chosen lattices, but this approach quickly becomes intractable: given an arbitrary rank $19$ lattice $L$, it is very dif\/f\/icult to f\/ind explicit representatives for all $L$-polarized K3 surfaces (see \cite{pfums} for some discussion of this problem). The work in this paper allows us to bypass such considerations in the presence of a known Picard--Fuchs equation. In a sense this is antipodal to the work of Fujino \cite{cbf2}, who computes the Hodge bundles of K3 surface f\/ibrations with only unipotent monodromy.

\subsection{Structure of this paper}

Section \ref{sec:background} begins with a brief discussion of the necessary background in Hodge theory and the theory of ordinary dif\/ferential equations, following \cite{lblefbc}, which we require in order to state the main results in Section \ref{sec:impres}. Section \ref{sec:background} concludes with a discussion of how these results are af\/fected by base change.

The remainder of the paper presents some applications of this theory. Firstly, in Section \ref{sect:Ell}, we apply the results of Section \ref{sec:background} to the case of elliptic f\/ibrations. Here we show how to compute the Hodge numbers of an elliptic surface over $\mathbb{P}^1$ directly from its Picard--Fuchs equation.

In Section \ref{sect:K3}, we move on to analyze $(1,1,1)$-type $\mathbb{R}$-VHS's over $\mathbb{P}^1$, which arise in the context of K3 f\/ibrations. We again compute their Hodge numbers using the corresponding Picard--Fuchs equations, proving a Hodge bundle formula for such f\/ibrations. In particular, we can use this to place constraints on the possible K3 f\/ibrations that can arise on a Calabi--Yau threefold, which allows us to recover an approximate form of a result from \cite{flpk3sm}.

Finally, in Section \ref{sect:CY3}, we will consider the case of Calabi--Yau threefold f\/ibrations over $\mathbb{P}^1$. Here we achieve our main goal and complete the computations of \cite{hca1pfcy3f} and \cite{hnccytls}. Finally, we conclude with a result that allows us to constrain the possible Calabi--Yau fourfolds which may admit f\/ibrations by the quintic mirror Calabi--Yau threefold. The results we use to do this apply in greater generality; we expect similar results to hold for any Calabi--Yau threefolds $X$ with $(1,1,1,1)$-type Hodge structure on $H^3(X,\mathbb{Q})$.

\section{Background and important results}\label{sec:background}

In this section we will develop the necessary background material and state a number of results that are necessary for the subsequent computations in Sections \ref{sect:Ell}, \ref{sect:K3} and \ref{sect:CY3}.

\subsection{Variations of Hodge structure} \label{sec:VHS}

Let us begin with a real variation of polarizable Hodge structure over a quasi-projective curve~${C}$. We denote this $\mathbb{R}$-VHS by $(\mathbb{V}, \mathscr{F}^\bullet, \nabla)$, where $\mathbb{V}$ is the underlying real local system, $\mathscr{F}^\bullet$ is the Hodge f\/iltration on $\mathscr{V} :=\mathbb{V} \otimes \mathscr{O}_C$ and $\nabla$ is the Gauss--Manin connection. We require that $\nabla$ satisfy Grif\/f\/iths transversality \cite{piam1,piam2,piam3}; in other words, $\nabla(\mathscr{F}^i)$ should be contained in $\mathscr{F}^{i+1} \otimes \Omega_C$.

Following Deligne \cite{edpsr}, we can canonically extend the bundle $\mathscr{V}$ to a bundle on the smooth completion $\overline{C}$ of $C$, as follows. Near a point $p \in \Delta = \overline{C} \setminus C$, we may choose a chart so that $p$ is the center of a disc $D$. Let $V_0$ be a f\/ibre of $\mathbb{V}$ near $p$ and suppose that monodromy around $p$ acts on $V_0$ by a quasi-unipotent transformation $T$. Def\/ine subspaces $W_{\alpha}$ of $V_0$ by
\begin{gather*}W_\alpha := \big\{ v \in V_0 \colon (T - \zeta_\alpha)^k v = 0\big\},\end{gather*}
where $\zeta_\alpha = e^{2\pi i \alpha}$ and $\alpha$ is chosen to be in the interval $[0,1)$. These vector spaces are zero for all but f\/initely many values of $\alpha \in [0,1)$, and we have a direct sum decomposition $V_0 = \bigoplus_{\alpha} W_{\alpha}$. The vector spaces $W_{\alpha}$ over the dif\/ferent points of the punctured disc $D \setminus \{p\}$ def\/ine a sub-bundle~$\mathbb{W}_{\alpha}$ of~$\mathbb{V}$.

Let $T_\alpha$ denote the action of monodromy on $\mathbb{W}_\alpha$; note that $\zeta_{\alpha}^{-1}T_{\alpha}$ is unipotent by construction. Def\/ine
\begin{gather*}N_{\alpha} := \frac{-1}{2\pi i}\log(T_{\alpha}),\end{gather*}
where the branch of the logarithm is chosen so that the unique eigenvalue of $N_\alpha$ is in the interval $[0,1)$. Then for any section $v$ of $\mathbf{e}^*\mathbb{W}_\alpha$, where $\mathbf{e}(z) := \exp(2 \pi i z)$ denotes the complex exponential, we may def\/ine
\begin{gather*}\tilde{v}(z) = \exp(2\pi i z (N_{\alpha} - \alpha))v(z),\end{gather*}
where $z \in \mathbb{H}$ is an element of the upper half-plane. Since $v(z+1) = T(v(z))$, we see that~$\tilde{v}(z)$ is invariant under the translation $z \mapsto z +1$, so $\tilde{v}$ def\/ines a holomorphic section of~$\mathscr{V}$ on the punctured disc~$D^*$. Deligne's canonical extension of~$\mathscr{V}$ is the bundle on $D$ which is the $\mathscr{O}_D$-module spanned by all such vectors over all $\alpha \in [0,1)$. This is denoted $\overline{\mathscr{V}}$.

The Hodge f\/iltration $\mathscr{F}^\bullet$ extends to a Hodge f\/iltration on $\overline{\mathscr{V}}$, which we denote by $\overline{\mathscr{F}}^\bullet$. The graded pieces of $\overline{\mathscr{F}}^\bullet$ will be denoted $\mathscr{E}^{p,\ell-p} = \overline{\mathscr{F}}^p / \overline{\mathscr{F}}^{p-1}$, where $\ell$ is the length of the Hodge f\/iltration.

Following \cite[Section 2]{lblefbc}, there is a parabolic f\/iltration of the f\/ibre $V_p$ of $\overline{\mathscr{V}}$ over each point $p \in \Delta$, given by $V^{\geq \beta} := \bigoplus_{\alpha \geq \beta} V_{\alpha}$, where $V_\alpha$ is the subspace of~$V_p$ spanned by local sections of~$\overline{\mathscr{V}}$ coming from sections of~$\mathbf{e}^* \mathbb{W}_\alpha$. The \emph{parabolic degree} of $\overline{\mathscr{V}}$ is given by
\begin{gather} \label{eq:pardeg} \deg_{\mathrm{par}} \overline{\mathscr{V}} := \deg \overline{\mathscr{V}} + \sum_{\substack{p \in \Delta \\ \alpha \in [0,1)}} \alpha \dim V_\alpha.\end{gather}

This parabolic f\/iltration on the f\/ibres of $\overline{\mathscr{V}}$ over points in $\Delta$ extends to a parabolic f\/iltration on the f\/ibres of $\mathscr{E}^{p,\ell-p}$; we denote its graded pieces by $\mathscr{E}^{p,\ell-p}_{\alpha}$. The connection $\nabla$ extends to a~connection on $\overline{\mathscr{V}}$ with logarithmic poles along $\Delta$ or, in other words,
\begin{gather*}\overline{\nabla} \colon \ \overline{\mathscr{V}} \longrightarrow \overline{\mathscr{V}} \otimes \Omega_C(\Delta).\end{gather*}
This connection is horizontal with respect to the Hodge f\/iltration $\overline{\mathscr{F}}^\bullet$, so we obtain morphisms of bundles
\begin{gather*}\theta_{i-1} \colon \ \mathscr{E}^{\ell-i,i} \longrightarrow \mathscr{E}^{\ell-i-1,i+1} \otimes \Omega_C(\Delta).\end{gather*}
This map also respects the local f\/iltration at each point $p \in \Delta$, i.e.,
\begin{gather*}\overline{\nabla} \colon \ V_p^{\geq \beta} \longrightarrow V_p^{\geq \beta} \otimes \Omega_C(\Delta).\end{gather*}

The following useful result appears as part of the proof of \cite[Theorem 6.1]{lblefbc}.

\begin{Lemma}\label{lemma:conj}
Assume that $\overline{C} = \mathbb{P}^1$. Then
\begin{gather*}\deg_{\mathrm{par}} \mathscr{E}^{p,q} = -\deg_{\mathrm{par}} \mathscr{E}^{q,p}.\end{gather*}
\end{Lemma}

The following two examples will be useful in later portions of this paper.

\begin{Example}[elliptic curves]\label{ex:Elc}
Let us take a degeneration of elliptic curves over the unit disc in $\mathbb{C}$, with central f\/ibre of Kodaira type $\mathrm{IV}^*$. In this case, one can check (see \cite[Table VI.2.1]{btes}) that the local monodromy action on the f\/irst integral cohomology of a general f\/ibre has matrix
\begin{gather*}\left( \begin{matrix} -1 & -1 \\ 1 & 0 \end{matrix}\right).\end{gather*}
This matrix has eigenvalues equal to $e^{2\pi i/3}$ and $e^{4\pi i /3}$. Therefore, the spaces $V_{1/3}$ and $V_{2/3}$ are both $1$-dimensional. We can ask how these spaces interact with the canonical extension of the Hodge f\/iltration. We know that the map $\nabla\colon \overline{\mathscr{F}}^0 \rightarrow \overline{\mathscr{F}}^{1} \otimes \Omega_D(0)$ respects the f\/iltration $\mathbb{C}^2 = V^{\geq 1/3} \supset V^{\geq 2/3}=\mathbb{C}$ and thus the image of $\nabla$ in $\overline{\mathscr{F}}^1 \otimes \Omega_D(0)$ must be $V^{\geq 2/3} = V_{2/3}$, since~$\nabla$ is not zero. Therefore $\mathscr{E}^{0,1}_{2/3} = \mathscr{E}^{0,1}$ and $\mathscr{E}^{0,1}_{1/3} = 0$. Similar arguments show that $\mathscr{E}^{1,0}_{1/3} = \mathscr{E}^{1,0}$ and $\mathscr{E}^{1,0}_{2/3} = 0$.

Using Kodaira's classif\/ication of germs of one parameter degenerations of elliptic curves, one can compute parabolic f\/iltrations on the Hodge bundles corresponding to all non-unipotent degenerations of elliptic curves in the same way. The result is given in Table~\ref{tab:elliptic}, which lists the values of $\alpha$ such that $\mathscr{E}^{1,0}_{\alpha} = \mathscr{E}^{1,0}$ and $\mathscr{E}^{0,1}_{\alpha} = \mathscr{E}^{0,1}$ for each type of singular f\/ibre.

\begin{table}\centering
\begin{tabular}{| c|c|c| }
 \hline
Fibre type & $\mathscr{E}^{1,0}$ & $\mathscr{E}^{0,1}$\tsep{2pt}
\\ \hline
$\mathrm{I}_n$ & $0$ & $0$ \\
$\mathrm{I}_n^*$ & $1/2$ & $1/2$ \\
$\mathrm{II}$ or $\mathrm{II}^*$ & $1/6$ & $5/6$ \\
$\mathrm{III}$ or $\mathrm{III}^*$ & $1/4$ & $3/4$\\
$\mathrm{IV}$ or $\mathrm{IV^*}$ & $1/3$ & $2/3$ \\
\hline
\end{tabular}
\caption{Values of $\alpha$ such that $\mathscr{E}^{p,q}_{\alpha} = \mathscr{E}^{p,q}$ for degenerations of elliptic curves.}\label{tab:elliptic}
\end{table}

Using this, the parabolic degrees of $\mathscr{E}^{1,0}$ and $\mathscr{E}^{0,1}$ can be determined easily once one knows their degrees. Furthermore, for an elliptic f\/ibration over $\mathbb{P}^1$, we may combine this with equation~\eqref{eq:pardeg} and Lemma \ref{lemma:conj} to show that $- \deg \mathscr{E}^{1,0} - \deg \mathscr{E}^{0,1}$ is equal to the number of f\/ibres around which monodromy is not unipotent. This is precisely \cite[Remark 2.4]{hnccytls}.
\end{Example}

\begin{Example}[$(1,1,1)$ variations of Hodge structure]\label{ex:K3}
Let us take a degeneration $\mathscr{X}$ of K3 surfaces over the complex unit disc $D$, so that the restriction of $\mathscr{X}$ to $D^* := D \setminus \{0\}$ is an $M$-polarized family of K3 surfaces, in the sense of \cite[Def\/inition~2.1]{flpk3sm}, for some rank~$19$ lattice~$M$. The transcendental lattices of the general f\/ibres give rise to a VHS of type $(1,1,1)$ over~$D^*$.

One can generally describe all possible monodromy matrices that can underlie such a $(1,1,1)$-VHS. Up to sign, they are just symmetric squares of parabolic and elliptic elements of $\mathrm{SL}_2(\mathbb{R})$; this comes from the fact that $\mathrm{O}(2,1)$ is, up to sign, a symmetric square of the group $\mathrm{SL}_2(\mathbb{R})$.

By \cite[Lemma 3.1]{hnccytls}, the monodromy of $\mathscr{X}$ around $0 \in D$ can be written in Jordan canonical form as either
\begin{gather*}\left(\begin{matrix} \lambda& 1 & 0 \\0 & \lambda & 1 \\ 0 & 0 & \lambda \end{matrix}\right) \qquad \text{or} \qquad \left(\begin{matrix} \lambda_1& 0 & 0 \\0 & \lambda_2 &0 \\ 0 & 0 & \lambda_3 \end{matrix}\right)\end{gather*}
for some $\lambda, \lambda_i \in \mathbb{C}$ roots of unity with $\sum\limits_{i=1}^3 \lambda_i \in \mathbb{Z}$. Cases that are relevant to our work \cite{cytfhrlpk3s} are
\begin{alignat*}{3}
& T_{\mathrm{un}} = \left(\begin{matrix} -1& 1 & 0 \\0 & -1& 1 \\ 0 & 0 & -1 \end{matrix}\right), \qquad && T_i = \left(\begin{matrix} \sqrt{-1}& 0 & 0 \\0 & - \sqrt{-1} & 0 \\ 0 & 0 & - 1\end{matrix}\right), & \\
& T_{\omega} = \left(\begin{matrix} - e^{2\pi i/3}& 0 & 0 \\0 & - e^{4\pi i/3} &0 \\ 0 & 0 & -1 \end{matrix}\right), \qquad && T_\mathrm{nod} = \left(\begin{matrix} -1 & 0 & 0 \\0 & 1 & 0 \\ 0 & 0 & 1 \end{matrix}\right), &
\end{alignat*}
and their powers.

As in the previous example, we can compute the f\/iltration on the degenerate Hodge structure from these matrices. The result is given in Table~\ref{tab:111}, which lists the values of $\alpha$ such that $\mathscr{E}^{2,0}_{\alpha} = \mathscr{E}^{2,0}$, $\mathscr{E}^{1,1}_{\alpha} = \mathscr{E}^{1,1}$, and $\mathscr{E}^{0,2}_{\alpha} = \mathscr{E}^{0,2}$ for each type of VHS.

\begin{table}
\centering
\begin{tabular}{ |c|c|c| c| }
 \hline
Monodromy matrix & $\mathscr{E}^{2,0}$ & $\mathscr{E}^{1,1}$ & $\mathscr{E}^{0,2}$\tsep{2pt} \\
\hline
$T_\mathrm{un}$ & $1/2$ & $1/2$ & $1/2$ \\
$T_\mathrm{un}^2$ & $0$ & $0$ & $0$ \\
$T_i$ & $1/4$ & $1/2$ & $3/4$ \\
$-T_i$ & $1/4$ & $0$ & $3/4$ \\
$T_i^2 $ & $1/2$ & $0$ & $1/2$ \\
$T_{\omega}$ & $1/6$ & $1/2$ & $5/6$ \\
$T_{\omega}^2$ & $1/3$ & $0$ & $2/3$ \\
$T_{\mathrm{nod}}$ & $0$ & $1/2$ & $0$ \\
\hline
\end{tabular}
\caption{Values of $\alpha$ such that $\mathscr{E}^{p,q}_{\alpha} = \mathscr{E}^{p,q}$ for $(1,1,1)$-VHS's.}\label{tab:111}
\end{table}

For a f\/ibration of K3 surfaces over $\mathbb{P}^1$, we may combine this with equation~\eqref{eq:pardeg} and Lemma~\ref{lemma:conj} to show that $-\deg \mathscr{E}^{2,0} - \deg \mathscr{E}^{0,2}$ computes the number of f\/ibres around which we have mono\-dromy matrix with Jordan normal form $T_{\mathrm{un}}$, $T_i$, $-T_i$, $T_i^2$, $T_\omega$ or $T_{\omega}^2$. If we adjust the def\/inition of~$\mathrm{II}$ in \cite[Section~3]{hnccytls} to be the set of points in~$\Delta$ whose monodromy matrices have at least two eigenvalues not equal to~$1$, then \cite[Remark~3.4]{hnccytls} holds.
\end{Example}

\subsection{Picard--Fuchs equation of a variation of Hodge structure}
Now let us take a look at how the Picard--Fuchs equation of a variation of Hodge structure arises; more details for everything in this section and the next may be found in \cite[Section 6.4]{lblefbc}. We start with the same basic setup as before: a~variation of Hodge structure $(\mathbb{V}, \mathscr{F}^\bullet, \nabla)$ over a~quasi-projective curve $C$ with its canonical extension to~$\overline{C}$. For the sake of simplicity, we will assume that $\overline{C} = \mathbb{P}^1$.

If we now choose a global section\footnote{More generally, if such a section does not exist, one may instead construct a Picard--Fuchs $D$-module; however this will not be relevant to our work.}~$\sigma$ of the Serre twist $\overline{\mathscr{F}}^\ell(i)$, where $\ell$ is the length of the Hodge f\/iltration and $i$ is some integer, then we can build a dif\/ferential equation associated to~$\sigma$. Taking a polarization $\langle -, - \rangle$ on~$\mathbb{V}$, we can assign to $\sigma$ a set of multivalued functions
\begin{gather*}s_v(t) = \langle \sigma, v \rangle,\end{gather*}
where $v$ is some f\/lat local section on $\mathbb{V}$ (or any f\/lat local section $v$ of $\mathscr{V} := \mathbb{V} \otimes \mathscr{O}_C$). The func\-tions~$s_v(t)$ give multivalued meromorphic functions on $\mathbb{P}^1$, with monodromy occurring only around points in $\Delta = \overline{C}\setminus C$. One can then produce a global dif\/ferential equation $L_\sigma$, of rank at most the rank of $\mathbb{V}$ on $C$, whose solution set is the set of functions~$s_v(t)$. The dif\/ferential operator $L_\sigma$ is a Fuchsian ODE (which will be discussed in the following section), called the \emph{Picard--Fuchs equation}.

The equation $L_\sigma$ is def\/ined as follows. For an appropriate choice of af\/f\/ine coordinate $t$ on $\overline{C} = \mathbb{P}^1$, we may assume that $\infty \in \Delta$, and therefore that the vector f\/ield~$d/dt$ is a global vector f\/ield on $C = \mathbb{P}^1 \setminus \Delta$. If~$\alpha$ and~$\beta$ are sections of~$\mathscr{V}$, then we have
\begin{gather*}\frac{d}{dt} \langle \alpha, \beta \rangle = \langle \nabla_{d/dt}\alpha, \beta \rangle \pm \langle \alpha, \nabla_{d/dt} \beta \rangle, \end{gather*}
where $\nabla_{d/dt} \alpha$ is def\/ined to be $\nabla(\alpha) \in \mathscr{V} \otimes \Omega_{C}$ paired with $d/dt$. Therefore,
\begin{gather*}\frac{ds_v(t)}{dt} = \langle \nabla_{d/dt} \sigma, v \rangle,\end{gather*}
since $v$ is chosen to be a f\/lat local section of~$\mathscr{V}$.

Since $\nabla^i_{d/dt}(\sigma)$ all live in $\mathscr{V}$, if $r = \rank \mathscr{V}$, then $\sigma,\nabla_{d/dt} \sigma,\dots, \nabla^{r}_{d/dt} \sigma$ must be linearly dependent over the f\/ield of meromorphic functions on $C$. Therefore, we can write an equation
\begin{gather*}\nabla^n_{d/dt}\sigma + f_1(t)\nabla^{n-1}_{d/dt} \sigma +\dots +f_{n}(t) \sigma = 0\end{gather*}
for some minimal $n \leq r$. Therefore, for any f\/lat local section $v$ of $\mathscr{V}$, we obtain the equation
\begin{gather*}\frac{d^ns_v(t)}{dt^n} + f_{1}(t)\frac{d^{n-1}s_v(t)}{dt^{n-1}} + \dots + f_{n}(t) s_v(t) = 0,\end{gather*}
by linearity of the operator $\langle -, v \rangle$. In other words, $s_v(t)$ is annihilated by the dif\/ferential operator
\begin{gather*}L_\sigma = \frac{d^n}{dt^n} + f_{1}(t) \frac{d^{n-1}}{dt^{n-1}} + \dots + f_n(t).\end{gather*}

On the other hand, to a rank $r$ ordinary dif\/ferential equation $L$ with regular singularities, we can assign to $L$ its solution sheaf $\mathrm{Sol}(L)$, which is a subsheaf of $\mathscr{O}_{C}$ spanned over $\mathbb{C}$ by the local solutions of $L$. The solution sheaf of $L_\sigma$ is precisely the subsheaf of $\mathscr{O}_C$ spanned locally by the functions $s_v(t)$.

\subsection{Fuchsian ODE's}
Let $L$ be a linear ordinary dif\/ferential equation in one variable, written as
\begin{gather*}\frac{d^n}{dt^n} + \sum_{i=1}^nf_i(t) \left( \frac{d}{dt}\right)^{n-i},\end{gather*}
where the $f_i$ are meromorphic functions in $t$. This equation can be rewritten in terms of the operator $\delta_t = t\frac{d}{dt}$ as
\begin{gather}\label{eq:delta}
\delta_t^n + \sum_{i}^{n}g_i(t) \delta_t^{n-i}
\end{gather}
for some meromorphic functions $g_1,\dots,g_n$. Indeed, one can check that
\begin{gather*}t^n \left(\frac{d}{dt}\right)^n = \delta_t(\delta_t-1) \cdots (\delta_t - n -1),\end{gather*}
so it is easy to translate between the two expressions for $L$.

A point $a \in \mathbb{P}^1$ is called \emph{nonsingular} if $f_i(t)$ is holomorphic at $a$, and is a \emph{regular singular point} if $(t-a)^i f_i(t)$ is holomorphic at $a$, for all $i \in \{1,\ldots,n\}$. Note that $0$ is a regular singular point if and only if the coef\/f\/icients $g_i(t)$ in equation~\eqref{eq:delta} are all holomorphic at $0$. A linear dif\/ferential equation of one variable with only regular singularities at all points in $\mathbb{P}^1$ is called a~\emph{Fuchsian ODE}. Regular singular points can be divided into \emph{apparent singularities}, around which monodromy acts as the identity on the solution sheaf, and \emph{actual singularities}, where it does not.

Let $L$ be a Fuchsian ODE with a regular singular point at $0 \in \mathbb{P}^1$. The polynomial
\begin{gather*}T^n + \sum_{i=1}^n g_i(0) T^{n-i}\end{gather*}
is called the \emph{indicial equation} of $L$ at $0$, and its roots are called the \emph{characteristic exponents} of~$L$ at~$0$. For an arbitrary point $a \in \mathbb{P}^1$, one may compute the indicial equation and characteristic exponents of $L$ at $a$ by making the variable change $t \mapsto t-a$. We will denote the characteristic exponents of~$L$ at $a$ by $\mu^{a}_1,\dots, \mu^a_n$; throughout the paper we will assume that the $\mu^{a}_i$ are ordered by magnitude, so that $\mu_1^a \leq \mu_2^a \leq \cdots \leq \mu_n^a$. Often, the data of the regular singular points $a_i$ and their characteristic exponents $\mu_i^{a_j}$ are arranged into a matrix known as the \emph{Riemann scheme} of~$L$, which looks like
\begin{gather*}\left\{ \begin{matrix}
a_1 & a_2 & \cdots & a_k \\ \hline
\mu^{a_1}_1 & \mu^{a_2}_1 & \cdots & \mu^{a_k}_1
\\
\vdots & \vdots & \ddots & \vdots \\
\mu^{a_1}_n & \mu^{a_2}_n & \cdots & \mu^{a_k}_n
\end{matrix} \right\}.\end{gather*}

Information about the monodromy of the solution sheaf of $L$ can be read of\/f from the Riemann scheme. For instance, if~$a$ is a regular singular point whose characteristic exponents are all integers, then monodromy around $a$ is unipotent, i.e., all of its eigenvalues are $1$. As a partial converse, if~$a$ is an apparent singularity, then the characteristic exponents at $a$ are all integers.

Finally, note that it also makes sense to compute the characteristic exponents at a nonsingular point $a$ of $L$. If one does so, one obtains $(\mu_1^a,\mu_2^a,\ldots,\mu_n^a) = (j,j+1,\ldots,j+n-1)$, for some integer $j$.

\begin{Remark}
It is often natural to scale the solutions to a dif\/ferential equation by a meromorphic function $h(t)$ on $\mathbb{P}^1$ to obtain a \emph{twisted} dif\/ferential equation, denoted $\prescript{h}{}L$. This operation acts on the characteristic exponents of $L$ by an overall scaling $\prescript{h}{}\mu_i^a = \mu_i^a + \mathrm{ord}_a(h)$, where $\mathrm{ord}_a(h)$ is the order of $h$ at $a$. Importantly, the dif\/ferences between the characteristic exponents are not af\/fected by the process of twisting. In the main results of this paper, we will only need to know the dif\/ferences between the characteristic exponents, so knowing $L_\sigma$ up to twist is enough.

This is important in applications, since one often only knows the Picard--Fuchs equation of a~variation of Hodge structure up to twist: for instance, this is all that the Grif\/f\/iths--Dwork method~\cite{opcri1,opcri2} for computing Picard--Fuchs equations of hypersurfaces guarantees. In Theorem~\ref{thm:ekmz}, only the dif\/ference in characteristic exponents is relevant; therefore, our Hodge bundle computations can be done with~$\prescript{h}{}L$ for any meromorphic function~$h$.
\end{Remark}

\begin{Remark}
One can also twist by multivalued functions $h$ on $\mathbb{P}^1$. A particular choice of (possibly) multivalued function $h$ can be made so that $\prescript{h}{}L$ has $f_{n-1}(t) = 0$. This is called the {\it projective normal form} of $L$. The process of twisting by a multivalued function \emph{does} af\/fect the dif\/ferences of characteristic exponents in a signif\/icant way. For this reason, we will avoid using the projective normal form in this paper.
\end{Remark}

\begin{Remark}
The characteristic exponents of $L$ at a point $a$ allow us to compute local solutions for $L$ near $a$. We remark that in general, the solutions of $Lf = 0$ around the point $a$ have the form $t^{\mu_i^a}P_i(t)$ for some holomorphic functions~$P_i(t)$, modulo terms with logarithmic singularities at $a$. For instance, if all~$\mu_i^a$'s are the same, then there is an ordered basis of local solutions~$\{s_i(t)\}$ around $a$ given inductively by
\begin{gather*}
s_i(t) = t^{\mu_i^a} P_i(t) + \sum_{j=1}^{i-1} \frac{s_{j}(t) \log t}{j!},
\end{gather*}
where the $P_i$ are holomorphic functions which do not vanish at $a$.

If none of the dif\/ferences between characteristic exponents are integers, then there are no terms with logarithmic singularities and the local solutions around $a$ all have the form $t^{\mu_i^a}P_i(t)$, for some holomorphic functions~$P_i(t)$. In this case, if~$L$ is the Picard--Fuchs equation of a~variation of Hodge structure $(\mathbb{V}, \mathscr{F}^\bullet, \nabla)$, then Deligne's canonical extension $\overline{\mathscr{V}}$ of~$\mathscr{V}$ is spanned by the set of local sections $\{t^{\lfloor \mu_i^a \rfloor} P_i(t)\}$ and the subspaces spanned by $t^{\lfloor \mu_i^a\rfloor}P_i(t)$ are the parabolic f\/iltrands~$V_{\alpha}$, where $ \alpha := \mu_i^a - \lfloor \mu_i^a \rfloor$. A more in-depth analysis of the Frobenius method shows that this is true in general.
\end{Remark}

\subsection{Important results}\label{sec:impres}

Here we state some important results which, in combination, will allow us to compute Hodge numbers of a variation of Hodge structure of $(1,1,\dots,1)$-type if we know the Picard--Fuchs equation $L_\sigma$. The f\/irst is a result of Eskin, Kontsevich, M\"oller, and Zorich \cite{lblefbc}, which says that the characteristic exponents of the Picard--Fuchs equation $L_\sigma$ corresponding to a variation of Hodge structure of $(1,1,\dots, 1)$-type can be used to compute the Kodaira--Spencer maps on the bundles $\mathscr{E}^{p,q}$.

\begin{Theorem}[{\cite[Lemma 6.3]{lblefbc}}] \label{thm:ekmz}
Let $(\mathbb{V},{\mathscr{F}}^\bullet, \nabla)$ be a polarized $\mathbb{R}$-VHS of $(1,1,\dots, 1)$-type over a Zariski open subset $\mathbb{P}^1 \setminus \Delta$, let $\overline{\mathscr{F}}^\bullet$ be its quasi-canonical extension to $\mathbb{P}^1$, let $\mathscr{E}^{p,q}$ be the graded pieces of $\overline{\mathscr{F}}^\bullet$, and let $L_\sigma$ be the Picard--Fuchs differential equation associated to a meromorphic section of $\sigma$ of $\mathscr{E}^{\ell,0}$. Then for $i < \ell/2$, the Kodaira--Spencer map $\theta_i \colon \mathscr{E}^{\ell-i,i} \rightarrow \mathscr{E}^{\ell-1-i,1+i} \otimes \Omega_{\mathbb{P}^1}(\Delta)$ is described as follows.
\begin{enumerate}\itemsep=0pt
\item[$1.$] If $p$ is a nonsingular point of $L_\sigma$, then $\theta_i$ is a local isomorphism at $p$.
\item[$2.$] If $p$ is an apparent singularity of $L_\sigma$, then $\theta_i$ has cokernel of length $\mu_{i+2}^p -\mu_{i+1}^p-1$.
\item[$3.$] If $p$ is an actual singularity of $L_\sigma$, then $\theta_i$ has cokernel of length $\lfloor \mu_{i+2}^p \rfloor - \lfloor \mu_{i+1}^p\rfloor$.
\end{enumerate}
\end{Theorem}

\begin{Remark}
 We record the observation that Eskin, Kontsevich, M\"oller, and Zorich use Theo\-rem~\ref{thm:ekmz} to compute the parabolic degrees of the Hodge bundles of hypergeometric variations of Hodge structure. Specif\/ically, they are interested in the 14 VHS's of $(1,1,1,1)$-type on $\mathbb{P}^1$ found by Doran and Morgan \cite{doranmorgan}. However, it seems as if the authors of \cite{lblefbc} were unaware of the fact that this result follows easily from the data in \cite{hnccytls}, in which the degrees of $\mathscr{E}^{3,0}$ and $\mathscr{E}^{2,1}$ are computed for all fourteen $(1,1,1,1)$-type variations of Hodge structure of hypergeometric type. We will see later that Theorem \ref{thm:ekmz} can be used to perform computations that neither the authors of \cite{hca1pfcy3f} nor \cite{hnccytls} were able to perform.
\end{Remark}

This result is useful because, if we know the parabolic f\/iltration of each bundle $\mathscr{E}^{p,q}$ and the Picard--Fuchs equation $L_\sigma$, then knowing the degrees of all maps $\theta_i$ allows us to easily compute the degrees of $\mathscr{E}^{p,q}$ when each $\mathscr{E}^{p,q}$ is a line bundle (i.e. when $L_{\sigma}$ underlies a variation of Hodge structure of weight $d$ and type $(1,1,\ldots,1)$).

Next we will state several results due to del Angel, M\"uller-Stach, van Straten and Zuo, and Hollborn and M\"uller-Stach, which appear in \cite{hca1pfcy3f,hnccytls}. These results say that, in the situations that interest us, we can compute Hodge numbers of $H^1(\mathbb{P}^1,j_*\mathbb{V})$ (where $j\colon C \hookrightarrow \overline{C} = \mathbb{P}^1$ denotes the inclusion) using the degrees of $\mathscr{E}^{p,q}$ and the local monodromy data of $\mathbb{V}$. According to Theorem~\ref{thm:ekmz}, all of this information can be obtained from $L_\sigma$.

\begin{Theorem}[\cite{hca1pfcy3f,hnccytls}]\label{theorem:msh}
Let $\mathbb{V}$ be a local system on $\mathbb{P}^1 \setminus \Delta$ which supports a variation of Hodge structure of weight $n$ and $(1,1,\dots,1)$-type. Then the Hodge number $h^{0,n+1}$ of $H^1(\mathbb{P}^1,j_*\mathbb{V})$ is given by
\begin{gather*}h^{0,n+1} = h^1\big(\mathbb{P}^1,\mathscr{E}^{0,n}\big) = h^0\big(\mathbb{P}^1, \mathscr{O}_{\mathbb{P}^1}\big({-}2-\deg \mathscr{E}^{0,n}\big)\big).\end{gather*}
\end{Theorem}

In the case where $n = 1$ or $2$ in the above theorem, then we can deduce the rest of the Hodge numbers of the parabolic cohomology $H^1(\mathbb{P}^1,j_*\mathbb{V})$ by Hodge symmetry and the following result, which generalizes the classical Riemann--Hurwitz formula (presumably this result is classically known, but the only reference that we know is~\cite{hca1pfcy3f}).

\begin{Proposition}[{\cite[Proposition~3.6]{hca1pfcy3f}}]\label{prop:rank} Let $\mathbb{V}$ be a local system on a quasi-projective curve~$C$ and let $j\colon C \hookrightarrow \overline{C}$ be the smooth completion. Let $\gamma_q$ be the local monodromy matrix around the point~$q$ in $\Delta := \overline{C} \setminus C$, acting on~$\mathbb{V}_p$ for some basepoint~$p$ in~$C$, and define $R(q) := \dim \mathbb{V}_p - \dim \mathbb{V}^{\gamma_q}_p$, where $\mathbb{V}^{\gamma_q}_p$ denotes the part of $\mathbb{V}_p$ fixed under $\gamma_q$. Then
\begin{gather*}h^1\big(\overline{C},j_*\mathbb{V}\big) = \sum_{q \in \Delta} R(q) + \big(2g\big(\overline{C}\big) -2\big)\rank \mathbb{V}.\end{gather*}
\end{Proposition}

 Finally, if $\mathbb{V}$ is a VHS of $(1,1,1,1)$-type, then Hollborn and M\"uller-Stach show how to compute the Hodge number~$h^{1,3}$ of the parabolic cohomology $H^1(\mathbb{P}^1,j_*\mathbb{V})$ from the degrees of its Hodge bundles.

\begin{Theorem}[{\cite[Theorem 4.3]{hnccytls}}]\label{thm:mash}
If $\mathbb{V}$ supports a variation of Hodge structure of type $(1,1,1,1)$ over a Zariski open subset $j\colon C \hookrightarrow \mathbb{P}^1$, then the Hodge number $h^{1,3}$ of $H^1(\mathbb{P}^1,j_*\mathbb{V})$ is given by
\begin{gather*}h^{1,3} = -2 + b -a + |\mathrm{II}| + |\mathrm{III}| + |\mathrm{IV}|,\end{gather*}
where $a = \mathrm{deg}(\mathscr{E}^{3,0})$ and $b = \deg (\mathscr{E}^{2,1})$. The sets $\mathrm{II}$ and $\mathrm{III}$ are the subsets of $\Delta = \mathbb{P}^1 \setminus C$ whose local monodromy matrices have Jordan normal forms
\begin{gather*}\left(\begin{matrix} 1 & 1 & 0 & 0 \\ 0 & 1 & 1 & 0 \\ 0 & 0 & 1 & 1 \\ 0 & 0 & 0 & 1 \end{matrix}\right), \qquad \text{and} \qquad \left(\begin{matrix} 1 & 1 & 0 & 0 \\ 0 & 1 & 0 & 0 \\ 0 & 0 & 1 & 1 \\ 0 & 0 & 0 & 1 \end{matrix}\right)\end{gather*}
respectively. The set~$\mathrm{IV}$ is the subset of $\Delta$ for which the monodromy matrix is strictly quasi-unipotent.
\end{Theorem}

\subsection{Base change and characteristic exponents}

The following instructive computation will be useful to us later. We want to understand how base change of a Fuchsian ODE af\/fects its Riemann scheme. More precisely, let us take a~Fuchsian ODE $L$ and a~map $g\colon \mathbb{P}^1 \rightarrow \mathbb{P}^1$. Pulling-back the solution sheaf of $L$ by $g$ yields a local sys\-tem~$g^*\mathrm{Sol}(L)$ and, in particular, another ODE $g^*L$, of the same rank as $L$, whose solution sheaf is~$g^*\mathrm{Sol}(L)$. Our goal in this section is to analyze the local exponents of $g^*L$. Away from the ramif\/ication points of $g$, the dif\/ferential equation $g^*L$ will have the same behaviour as $L$ at the corresponding point. So we can restrict ourselves to studying the local behaviour of $L$ at ramif\/ication points, and this can be done in the power series ring $\mathbb{C}[[t]]$. Therefore, our question reduces to a question of base change of a Fuchsian ODE along the map $s^k = t$, for a positive integer $k$.

The following computation appears for rank 2 ODEs in \cite{pfummmmm}.

\begin{Proposition}[{\cite[Lemma 3.21]{pfummmmm}}]
Base change of order $k$ at a point $p \in \mathbb{P}^1$ affects the characteristic exponents of $L$ by multiplying everything by $k$.
\end{Proposition}

\begin{proof}
Assume that $w(t)$ is a local solution to $Lw = 0$. By the chain rule, we have that
\begin{gather*}\frac{dw}{ds}\big(s^k\big) = ks^{k-1}\frac{dw}{dt}\big(s^k\big),\end{gather*}
therefore
\begin{gather*}\delta_t w\big(s^k\big) = s^k \frac{dw}{dt}\big(s^k\big) = \frac{s}{k} \frac{dw}{ds}\big(s^k\big) = \frac{\delta_s}{k}w\big(s^k\big).\end{gather*}
Thus if $w$ is a solution to the dif\/ferential equation
\begin{gather*}\delta_t^n+ \sum_{i=1}^{n}f_i(t) \delta_t^{n-i},\end{gather*}
then $w(s^k)$ solves
\begin{gather*}{\delta_s^n} + \sum_{i=1}^n k^i f_i\big(s^k\big) {\delta_s^{n-i}}.\end{gather*}
Suppose that the indicial equation of $L$ at $0$ is given by
\begin{gather*}\prod_{i=1}^n (T - \mu_i),\end{gather*}
for $\mu_i$ the local exponents of $L$ at $0$. Then the indicial equation of the dif\/ferential equation in terms of~$s$ is written as
\begin{gather*}\prod_{i=1}^n (T - k\mu_i).\end{gather*}
This proves the proposition.
\end{proof}

\begin{Corollary}
After base change, the ramification points of $g\colon\mathbb{P}^1 \to \mathbb{P}^1$ become apparent singularities of $g^*L$. If $p$ is a point with characteristic exponents all $0$, then for any point $q \in g^{-1}(p)$, the characteristic exponents of $g^*L$ at $q$ are all $0$.
\end{Corollary}

Combining this with Theorem \ref{thm:ekmz}, we can say quite a bit about how the Hodge numbers of a variation of Hodge structure of type $(1,1,\dots, 1)$ are altered by base change.

\begin{Remark}
We will take this opportunity to emphasize that performing base change on a~variation of Hodge structure and base change on a linear dif\/ferential equation may not in fact be compatible. In other words, there is a natural variation of Hodge structure on~$\mathbb{P}^1$ underlying~$g^*\mathbb{V}$ with appropriate canonical extensions at the points in $g^{-1}(\Delta)$. The Hodge bundles $\mathscr{E}^{p,q}_g$ of this variation of Hodge structure may not agree with the pull-backs $g^*\mathscr{E}^{p,q}$ of the Hodge bundles associated to $\mathbb{V}$. However, the monodromy representation and period maps are correct, so if $L_g$ is the Picard--Fuchs system of the pulled back variation of Hodge structure, then $g^*L = \prescript{h}{}L_g$ for some meromorphic function $h$. Thus the dif\/ferences between the characteristic exponents of~$g^*L$ are the same as for $L_g$, so we can apply Theorem~\ref{thm:ekmz} to compute the degrees of~$\mathscr{E}^{p,q}_g$.

The underlying reason behind this incompatibility is that Deligne's canonical extension does not commute with base change. Indeed, Deligne's canonical extension of $g^*\mathbb{V}$ chooses a branch of the logarithm so that the eigenvalues of the maps $N_{\alpha}$ at each point of $g^{-1}(\Delta)$ lie in the interval~$[0,1)$ (see Section~\ref{sec:VHS}). However base change acts additively on these eigenvalues, so may take them out of this interval; the pulled-back variation of Hodge structure may thus correspond to a dif\/ferent choice of extension of~$g^*\mathbb{V}$.
\end{Remark}

\section{Families of elliptic curves} \label{sect:Ell}
In the remainder of this paper we will study some applications of these results in various settings. We begin with the case of a family of elliptic curves.

Let $L$ be a rank 2 ODE which underlies a real variation of Hodge structure of weight $1$. We will show that if $L$ is the Picard--Fuchs equation of a family of elliptic curves $S$, then we can determine the geometry of $S$ from $L$.

First, we will restate Theorem \ref{thm:ekmz} in the case of a $(1,1)$-VHS.

\begin{Corollary}\label{thm:ell}
Let $L$ be a rank $2$ Fuchsian ODE which underlies an $\mathbb{R}$-VHS of weight $1$, with regular singular points at $\Delta$ in $\mathbb{P}^1$. Then the map $\theta\colon \mathscr{E}^{1,0} \rightarrow \mathscr{E}^{0,1} \otimes \Omega_{\mathbb{P}^1}(\Delta)$ has the following structure.
\begin{itemize}\itemsep=0pt
\item[$1)$] if $p$ is a nonsingular point of $L$, then $\theta$ is an isomorphism at $p$,
\item[$2)$] if $L$ has an apparent singularity at $p$, then $\theta$ has cokernel of length $\mu^p_2- \mu^p_1 - 1$ at $p$,
\item[$3)$] if $p \in \Delta$, then $\theta$ has cokernel of length $\left \lfloor{\mu^p_2}\right \rfloor - \left \lfloor{\mu^p_1}\right \rfloor$.
\end{itemize}
\end{Corollary}

The following theorem follows easily.

\begin{Theorem}\label{thm:elldeg}
Suppose that $L$ underlies a family of elliptic curves over $\mathbb{P}^1$. Let $a_L$ denote the number of points in $\Delta$ at which the local monodromy of $L$ is strictly quasi-unipotent, and let $\Delta_a$ denote the set of points at which $L$ has an apparent singularity. Then
\begin{gather*}\deg \mathscr{E}^{0,1} = \frac{1}{2} \bigg( 2 -a_L + \sum_{p \in \Delta}\big( \lfloor \mu_2^p\rfloor - \lfloor \mu_1^p \rfloor - 1\big) + \sum_{p \in \Delta_a}\big( \mu_2^p - \mu_1^p - 1\big)\bigg).\end{gather*}
\end{Theorem}
\begin{Remark} \label{rem:ell} If we note that the characteristic exponents $\mu_i^p$ are integers when~$p$ is nonsingular or an apparent singularity, and that $\mu_2^p = \mu_1^p + 1$ when $p$ is a non-singular point, then we can simplify the formula in Theorem~\ref{thm:elldeg} to
\begin{gather*}\deg \mathscr{E}^{0,1} = \frac{1}{2} \bigg( 2 -a_L + \sum_{p \in \mathbb{P}^1}\big( \lfloor \mu_2^p\rfloor - \lfloor \mu_1^p \rfloor - 1\big)\bigg).\end{gather*}
However, the form given in Theorem \ref{thm:elldeg} is simpler to work with from a computational standpoint, as both sums are clearly f\/inite.
\end{Remark}
\begin{proof}[Proof of Theorem \ref{thm:elldeg}]
By Example \ref{ex:Elc}, we see that
\begin{gather*}\deg \mathscr{E}^{1,0} + \deg \mathscr{E}^{0,1} = -a_L,\end{gather*}
and Corollary \ref{thm:ell} gives that the map $\theta$ has cokernel of length
\begin{gather*}\sum_{p \in \Delta_a} \big(\mu_2^p - \mu_1^p - 1\big) + \sum_{p \in \Delta} \big(\lfloor \mu^p_2 \rfloor - \lfloor \mu^p_1 \rfloor\big).\end{gather*}
Since $\deg (\mathscr{E}^{0,1} \otimes \Omega_{\mathbb{P}^1}(\Delta)) = |\Delta| -2 +\deg{\mathscr{E}^{0,1}}$, it follows that
\begin{gather*}\deg \mathscr{E}^{0,1} + |\Delta| - 2 = -\deg \mathscr{E}^{0,1} -a_L + \sum_{p \in \Delta_a} \big(\mu_2^p - \mu_1^p - 1\big) + \sum_{p \in \Delta} \big(\lfloor \mu^p_2 \rfloor - \lfloor \mu_1^p \rfloor\big).\end{gather*}
Thus
\begin{gather*}\deg \mathscr{E}^{0,1} = \frac{1}{2} \bigg(2 -a_L + \sum_{p \in \Delta_a} \big(\mu_2^p - \mu_1^p - 1\big) + \sum_{p \in \Delta} \big(\lfloor \mu^p_2 \rfloor - \lfloor \mu_1^p \rfloor-1\big)\bigg)\end{gather*}
as required.
\end{proof}

The term $(\mu_2^p - \mu_1^p-1)$ should be thought of as an invariant which counts the ramif\/ication of the $j$-function at $p \in \Delta_a$; Doran makes this precise in~\cite{pfummmmm}.

\begin{Remark}From this, we may use Theorem~\ref{theorem:msh} to compute $h^{0,2}$ for the parabolic cohomology of our elliptic f\/ibration. In~\cite{inses}, Cox and Zucker show that this~$h^{0,2}$ computes the dimension of a particular space of modular forms. The computations in this section can be thought of as giving a way to compute the dimension of this space of modular forms using the Picard--Fuchs equation of the elliptic surface.
\end{Remark}

\section{Families of K3 surfaces} \label{sect:K3}

Next we turn our attention to families of K3 surfaces, as studied in Example~\ref{ex:K3}. Theorem~\ref{thm:ekmz} can be rephrased in the context of the corresponding $(1,1,1)$-variations of Hodge structure, as follows. Recall that in this case we have the Kodaira--Spencer maps $\theta_0\colon \mathscr{E}^{2,0} \rightarrow \mathscr{E}^{1,1} \otimes \Omega_{\mathbb{P}^1}(\Delta)$ and $\theta_1\colon \mathscr{E}^{1,1} \rightarrow \mathscr{E}^{0,2} \otimes \Omega_{\mathbb{P}^1}(\Delta)$.

\begin{Corollary}\label{k3cor}
If $L$ is a Picard Fuchs equation corresponding to a $(1,1,1)$-type variation of Hodge structure, then
\begin{itemize}\itemsep=0pt
\item[$1)$] if $p$ is a nonsingular point of $L$, then $\theta_0$ is an isomorphism,
\item[$2)$] if $p$ is an apparent singularity of $L$, then $\theta_0$ has cokernel of length $\mu_2^p - \mu_1^p - 1$ at $p$,
\item[$3)$] if $p$ is a regular singular point of $L$, then $\theta_0$ has cokernel of length $\lfloor \mu_2^p \rfloor - \lfloor \mu_1^p \rfloor$.
\end{itemize}
\end{Corollary}

Let $U = \mathbb{P}^1 \setminus \Delta$ denote a Zariski open set and let $j\colon U \hookrightarrow \mathbb{P}^1$ denote the embedding. Suppose that $\mathcal{X} \to U$ is an $M$-polarized family of K3 surfaces, in the sense of \cite[Def\/inition 2.1]{flpk3sm}, for some rank $19$ lattice $M$. The transcendental lattices of the f\/ibres of $\mathcal{X}$ def\/ine a VHS $\mathbb{V}$ of type $(1,1,1)$ over $U$; we assume that all of the monodromy matrices of this VHS take the forms $T_x$ considered in Example~\ref{ex:K3}. Finally, let~$L$ denote the Picard--Fuchs equation underlying this $(1,1,1)$-VHS.

By Lemma \ref{lemma:conj}, it is easy to see that $\deg_\mathrm{par} \mathscr{E}^{1,1} = 0$, so the degree of $\mathscr{E}^{1,1}$ can be computed easily using equation \eqref{eq:pardeg}: indeed, if $a_{1/2}$ is the number of points $p$ at which monodromy is of type $T_\mathrm{un}$, $T_i$, $T_\omega$ or $T_\mathrm{nod}$ (see Example \ref{ex:K3}), then $\deg \mathscr{E}^{1,1} = -\frac{1}{2}a_{1/2}$. We let $a_{f}$ denote the number of points whose monodromy matrices look like $T_\mathrm{un}$, $T_i$, $-T_i$, $T_i^2$, $T_\omega$ and $T_\omega^2$. Finally, let~$\Delta_a$ denote the set of points at which $L$ has an apparent singularity.

\begin{Proposition}
If $L$ underlies a family of K3 surfaces $\mathcal{X} \to U$ as above, then we have
\begin{gather*}\deg\big(\mathscr{E}^{0,2}\big) = 2 + \frac{1}{2} a_{1/2} - a_f + \sum_{p \in \Delta_a} \big(\mu_2^p - \mu_1^p - 1\big) + \sum_{p \in \Delta} \big(\lfloor \mu^p_2 \rfloor - \lfloor \mu_1^p \rfloor-1\big).\end{gather*}
\end{Proposition}
\begin{proof}This is proved by a similar calculation to Theorem~\ref{thm:elldeg}.\end{proof}

\begin{Remark} As in the case of elliptic curves, we may simplify this formula to the shorter, but less transparent, form
\begin{gather*}\deg\big(\mathscr{E}^{0,2}\big) = 2 + \frac{1}{2} a_{1/2} - a_f + \sum_{p \in \mathbb{P}^1} \big(\lfloor \mu^p_2 \rfloor - \lfloor \mu_1^p \rfloor-1\big).\end{gather*}
\end{Remark}

Using this, one can easily use Theorem \ref{theorem:msh} to compute $h^{0,3}$ of the Hodge structure on $H^1(\mathbb{P}^1,j_*\mathbb{V})$.

\begin{Corollary}\label{cor:computeK3}
The Hodge number $h^{0,3}$ of the Hodge structure on $H^1(\mathbb{P}^1,j_*\mathbb{V})$ is
\begin{gather*}h^{0,3} = h^0\bigg(\mathbb{P}^1, \mathscr{O}_\mathbb{P}^1\bigg( a_f -4- \frac{1}{2}a_{1/2} - \sum_{p \in \Delta_a} \big(\mu_2^p - \mu_1^p - 1\big) - \sum_{p \in \Delta} \big(\lfloor \mu^p_2 \rfloor - \lfloor \mu_1^p \rfloor-1\big)\bigg)\bigg).\end{gather*}
\end{Corollary}

Now we may ask ourselves: when can such a f\/ibration by K3 surfaces have a Calabi--Yau threefold total space?

\begin{Definition}
A weight $n$ Hodge structure with Hodge numbers $h^{p,q}$ is called \emph{Calabi--Yau} if $h^{n,0}=1$.
\end{Definition}

If a variety $X$ is Calabi--Yau of dimension $n$, then the Hodge structure on~$H^n(X,\mathbb{Q})$ is certainly Calabi--Yau, but the converse is not true. For example, look at any blow-up of an honest Calabi--Yau variety, or a Kulikov surface.

According to Corollary \ref{cor:computeK3}, a family of K3 surfaces $\mathcal{X} \to U$ has Calabi--Yau Hodge structure on $H^1(\mathbb{P}^1,j_*\mathbb{V})$ if and only if
\begin{gather}\label{CYK3}
a_f = 4+ \frac{1}{2}a_{1/2} + \sum_{p \in \Delta_a} \big(\mu_2^p - \mu_1^p - 1\big) + \sum_{p \in \Delta} \big(\lfloor \mu^p_2 \rfloor - \lfloor \mu_1^p \rfloor-1\big).
\end{gather}
By the degeneration of the Leray spectral sequence \cite[Section 15]{htdcl2cpm}, for any compactif\/ication $\overline{\mathcal{X}}$ of $\mathcal{X}$, this parabolic cohomology group is a direct summand of the Hodge structure on $H^3(\overline{\mathcal{X}},\mathbb{Q})$, and its complement has no $(3,0)$ or $(0,3)$ part (see \cite{cytfhrlpk3s}). Thus, if equation \eqref{CYK3} does not hold, then $\mathcal{X} \to \mathbb{P}^1$ has no compactif\/ication which is a Calabi--Yau threefold. In \cite{cytfhrlpk3s} we will study the converse problem in specif\/ic examples.

\subsection{Base change of a hypergeometric family}
We illustrate these results with an example regarding families of K3 surfaces and their pullbacks. Let us take the family of mirror quartic K3 surfaces written as the compactif\/ications in $\mathbb{P}^3$ of the f\/ibres of the Laurent polynomial
\begin{gather*}f(x,y,z) = \frac{(x + y + z + 1)^4}{xyz}.\end{gather*}
Narumiya and Shiga \cite[Section 5]{mmfk3sis3drp} show that this family has Picard--Fuchs equation which is hypergeometric and may be written as
\begin{gather*}
\delta^3 + t\left(\delta + \tfrac{1}{4}\right)\left(\delta + \tfrac{1}{2}\right)\left(\delta + \tfrac{3}{4}\right).
\end{gather*}
The local monodromies around its singular points are of type $T_{\mathrm{un}}^2$, $T_i$, and $T_\mathrm{nod}$, over $0$, $\infty$, and~$1$ respectively, so the Riemann scheme of this dif\/ferential operator is
\begin{gather*}\left\{\begin{matrix} 0 & \infty & 1 \\\hline
0 & 1/4 & 0 \\
0 & 1/2 & 1/2\\
0 & 3/4 & 1
 \end{matrix}\right\}.\end{gather*}

Now consider the base change of this local system along a map $g\colon \mathbb{P}^1 \rightarrow \mathbb{P}^1$. Let~$\mathscr{E}^{i,j}_g$ denote the appropriate Hodge bundles of the pulled-back local system $g^* \mathbb{V}$ and let $h^{i,j}_g$ be the Hodge numbers of its parabolic cohomology. Let $d$ be the degree of $g$, and set $k$, $\ell$, $m$ to be the numbers of points over~$0$,~$\infty$ and~$1$ respectively. Let
\begin{gather*}r = \sum_{p \in \mathbb{P}^1 \setminus \{0,1,\infty \} }(e_{p}-1)\end{gather*}
denote the degree of ramif\/ication of $g$ away from $\{0,1,\infty\}$. We will write $[y_1,\dots, y_\ell]$ for the partition of $d$ encoding the ramif\/ication prof\/ile over $\infty$. Finally, def\/ine
\begin{gather*}D_g := a_f - a_{1/2} - 4 - \sum_{p \in \Delta_a} (\mu_2^p - \mu_1^p - 1) - \sum_{p \in \Delta} \big(\lfloor \mu^p_2 \rfloor - \lfloor \mu_1^p \rfloor-1\big),\end{gather*}
where all terms in this expression are computed for the pulled-back local system~$g^*\mathbb{V}$. Then we have the following proposition.

\begin{Proposition} \label{prop:mirrorquartics}
With assumptions and notation as above,
\begin{gather*}D_g = a_f -2 + \sum_{i=1}^\ell \left\lfloor \frac{y_i}{4}\right\rfloor.\end{gather*}
In particular, $h^{0,3}_g = h^0\big(\mathbb{P}^1, \mathscr{O}_{\mathbb{P}^1}(D_g)\big)$.
\end{Proposition}
\begin{proof} It follows from the def\/inition that $a_f$ is equal to the number of points over $0$ with ramif\/ication of order not divisible by~$4$. Let $\ell_{\mathrm{odd}}$ and $m_{\mathrm{odd}}$ denote the number of points over $0$ and $1$, respectively, where the order of ramif\/ication is odd. Then
\begin{gather*}a_{1/2} = m_\mathrm{odd} + \ell_\mathrm{odd}.\end{gather*}

Apparent singularities appear only at ramif\/ication points of $g$ away from points over $\{0,1,\infty\}$, and each ramif\/ication point has $\mu^p_2 - \mu^p_1 - 1$ equal to $e_p -1$. Therefore
\begin{gather*}\sum_{\Delta_a}\big(\mu^p_2 - \mu^p_1 - 1\big) = \sum_{p \in \mathbb{P}^1 \setminus \{0,1,\infty \} } (e_p -1) = r.\end{gather*}

We split the set $\Delta$ into three components, corresponding to points over $0$, $1$, and $\infty$ respectively, and denoted $\Delta_0$, $\Delta_1$, and $\Delta_\infty$. The corresponding sums are
\begin{gather*}
\sum_{p \in \Delta_0} \big(\lfloor \mu_2^p \rfloor - \lfloor \mu_1^p \rfloor -1\big) = \frac{1}{2}(d - \ell_\mathrm{odd}) - \sum_{i=1}^\ell \big( \lfloor \frac{y_i}{4} \rfloor\big) - \ell, \\
\sum_{p \in \Delta_1} \big(\lfloor \mu_2^p \rfloor - \lfloor \mu_1^p \rfloor -1\big) = \frac{1}{2}(d- m_\mathrm{odd}) - m,\\
\sum_{p \in \Delta_\infty} \big(\lfloor \mu_2^p \rfloor - \lfloor \mu_1^p \rfloor -1\big) = k.
\end{gather*}
Substituting everything into $D_g$, we obtain
\begin{gather*}D_g = a_f - 4 -d - r + k+ \ell + m + \sum_{i=1}^\ell \left\lfloor \frac{y_i}{4} \right\rfloor.\end{gather*}

Now, since $g$ is a map from $\mathbb{P}^1$ to $\mathbb{P}^1$, the Riemann--Hurwitz formula gives
\begin{gather*}2d-2 = \sum_{p} (e_p-1) = r + (d-\ell) + (d- m) + (d-k),\end{gather*}
which implies that $ k + \ell + m -d-r = 2$. Making this substitution, we f\/ind that
\begin{gather*}D_g = a_f-2 +\sum_{i=1}^\ell \left\lfloor \frac{y_i}{4} \right\rfloor,\end{gather*}
which proves the proposition.
\end{proof}

As a corollary, we can approximately recover \cite[Proposition 2.4]{cytfmqk3s} (see also \cite{cytfhrlpk3s} for some further applications of this result).
\begin{Corollary}
The Hodge structure of $g^*\mathbb{V}$ is Calabi--Yau if and only if one of the following two conditions holds:
\begin{enumerate}\itemsep=0pt
 \item[$1)$] the number $\ell$ is equal to $2$ and $1 \leq y_1,y_2\leq 4$; or
 \item[$2)$] the number $\ell$ is equal to $1$ and $5 \leq y_1 \leq 8$.
\end{enumerate}
\end{Corollary}
\begin{proof}
By Corollary \ref{cor:computeK3} and Proposition \ref{prop:mirrorquartics}, we may assume that
\begin{gather*}a_f + \sum_{i=1}^\ell \lfloor \frac{y_i}{4} \rfloor = 2.\end{gather*}
Therefore, we have at most two points over $0$. If we have two points over $0$, then they must both have $1 \leq y_i \leq 4$. If there is a single point over $0$, then $5 \leq y_1 \leq 8$.
\end{proof}

\section{Families of Calabi--Yau threefolds} \label{sect:CY3}

In this section, we will f\/irst apply the results of Section \ref{sec:impres} to complete the computations of~\cite{hca1pfcy3f} and~\cite{hnccytls}. Then we will determine the Hodge numbers of some specif\/ic fourfolds f\/ibred by quintic mirror Calabi--Yau threefolds.

\subsection{Inhomogeneous Picard--Fuchs equations and Calabi--Yau threefolds}
In this section, we will perform some computations related to work of del Angel, M\"uller-Stach, van Straten, and Zuo~\cite{hca1pfcy3f}. These authors were interested in determining situations in which families of Calabi--Yau threefolds admit normal functions; here a \emph{normal function} is an algebraic cycle on each f\/ibre of a family of varieties, which is homologous to zero on all f\/ibres.

In particular, if we have some algebraic $(2,2)$-class $Z$ on a family of Calabi--Yau threefolds, whose restriction to each f\/ibre is homologous to zero, then $Z$ determines a multivalued function~$\Phi_Z$ on the base of the f\/ibration. From such a normal function, one obtains an inhomogeneous extension of the Picard--Fuchs equation $L_\sigma$ which is satisf\/ied by $\Phi_Z$. In other words,
\begin{gather*}L_\sigma \Phi_Z = g_Z,\end{gather*}
for some holomorphic function $g_Z$. Morrison and Walcher \cite{dbnf, ehalaots} studied a specif\/ic such normal function on the Fermat family of quintics, and showed that $\Phi_Z$ can be interpreted as the domain wall tension of certain $D$-branes.

The computations in \cite{hca1pfcy3f} aimed to determine whether such normal functions exist on other $1$-parameter families of Calabi--Yau threefolds. Evidence for the existence of normal functions can be found by determining whether the parabolic cohomology of $L_\sigma$ has classes of type $(2,2)$ or not. They looked at the 14 hypergeometric examples of Doran and Morgan~\cite{doranmorgan} and their pullbacks along maps $[t:s] \mapsto [t^d: s^d]$, for certain values of $d$, in an ef\/fort to determine whether, after such base change, the $14$ families could admit interesting normal functions.

In both \cite{hca1pfcy3f} and the subsequent work of Hollborn and M\"uller-Stach \cite{hnccytls}, the authors were not able to compute the degrees of the bundles $\mathscr{E}^{2,1}$ in many situations, and therefore were not able to determine the Hodge numbers of the parabolic cohomology groups that they were interested in. By applying the results of Section \ref{sec:impres}, we are now able to f\/inish their computations.

\begin{Example}
Let's compute the Hodge numbers of the VHS obtained from the second hypergeometric example in the list of Doran and Morgan \cite{doranmorgan}. This example is obtained in the following way. Begin with the hypergeometric operator
\begin{gather*}\delta^4 - t\left(\delta +\frac{1}{10} \right)\left(\delta +\frac{3}{10} \right)\left(\delta +\frac{7}{10} \right)\left(\delta +\frac{9}{10} \right).\end{gather*}
This operator underlies a $(1,1,1,1)$-type rational VHS, and has local monodromy matrices
\begin{gather*}T_0 = \left( \begin{matrix} 1& 1 & 0 & 0 \\ 0 & 1 & 1 & 0 \\ 0 & 0 & 1& 1 \\ 0 & 0 & 0 & 1 \end{matrix}\right),\qquad T_1 = \left( \begin{matrix} 1& 0 & 0 & 0 \\ 0 & 1 & 0& 0 \\ 0 & 0 & 1& 1 \\ 0 & 0 & 0 & 1 \end{matrix}\right),\qquad T_\infty = \left( \begin{matrix} \zeta_{10} & 0 & 0 & 0 \\ 0 & \zeta_{10}^3 & 0 & 0 \\ 0 & 0 & \zeta_{10}^7 & 0 \\ 0 & 0 & 0 & \zeta_{10}^9 \end{matrix}\right),\end{gather*}
where $\zeta_{10}$ is a primitive tenth root of unity. Its Riemann scheme is
\begin{gather*}\left\{
\begin{matrix}
0 & 1 & \infty \\ \hline
0 & 0 & 1/10 \\
0 & 1 & 3/10 \\
0 & 1 & 7/10 \\
0 & 2 & 9/10
\end{matrix} \right\}.\end{gather*}

Let us take the base change of this dif\/ferential equation under the map of degree $5$ which is totally ramif\/ied over $0$ and $\infty$. The new dif\/ferential equation has seven singular points and Riemann scheme
\begin{gather*}\left\{
\begin{matrix}
0 & t^5 -1 = 0 & \infty \\ \hline
0 & 0 & 1/2 \\
0 & 1 & 3/2 \\
0 & 1 & 7/2 \\
0 & 2 & 9/2
\end{matrix} \right\}.\end{gather*}
The points over $1$ all have monodromy of type $T_1$, the point over $0$ has monodromy of type~$T_0$, and the point over $\infty$ has monodromy matrix equal to $-\mathrm{Id}_4$. Therefore, the only point with nontrivial parabolic f\/iltration is $\infty$ and the only f\/iltrands are of weight~$1/2$. Thus equation~\eqref{eq:pardeg} and Lemma~\ref{lemma:conj} give
\begin{gather*}\deg \mathscr{E}^{1,2} + \deg \mathscr{E}^{2,1} = -1.\end{gather*}

By Theorem \ref{thm:ekmz}, the map $\theta_1\colon \mathscr{E}^{2,1} \rightarrow \mathscr{E}^{1,2} \otimes \Omega_{\mathbb{P}^1}(\Delta)$ has cokernel of length $2$. Therefore,
\begin{gather*}\deg \mathscr{E}^{2,1} - \deg \mathscr{E}^{1,2} = 3.\end{gather*}
Hence $\deg \mathcal{E}^{2,1} = 1$. This gives the correct $b$ value in Theorem \ref{thm:mash}, which in turn allows us to complete the corresponding row of the table in \cite[Section 4]{hnccytls}; we obtain that the Hodge numbers of~$g^*L$ are $(0,1,2,1,0)$.

Similarly, if we take $g$ to be the degree $10$ map ramif\/ied completely at $0$ and $\infty$, then
\begin{gather*}\deg \mathscr{E}^{1,2} + \deg \mathscr{E}^{2,1} = 0.\end{gather*}
In this case $\theta_1$ has cokernel of length $4$, so
\begin{gather*}\deg \mathscr{E}^{2,1} - \deg \mathscr{E}^{1,2} = 6,\end{gather*}
and therefore $\deg \mathscr{E}^{2,1} =3$. Once again, this allows us to complete the corresponding row of the table in \cite[Section 4]{hnccytls}; we obtain that the Hodge numbers of $g^*L$ are $(0,1,3,1,0)$.
\end{Example}

Using this, we may complete the calculations left unf\/inished by \cite{hca1pfcy3f} and \cite{hnccytls}. Table~\ref{tab:msh} is a~reproduction of the table from \cite[Section~4]{hnccytls}, with all missing pieces completed; our new results are highlighted in gray. Our notation is as follows. \#~is the case number from the table in \cite[Section~4]{hnccytls}, and the numbers $(\alpha_1,\alpha_2,\alpha_3,\alpha_4)$ are the numbers determining the corresponding hypergeometric dif\/ferential equations
\begin{gather*}L = \delta^4 - t(\delta + \alpha_1) (\delta + \alpha_2) (\delta + \alpha_3) (\delta + \alpha_4).\end{gather*}
The number $d$ indicates that we have pulled back by the covering $[t:s] \mapsto [t^d:s^d]$, to obtain a VHS over $\mathbb{P}^1$ with Hodge bundles $\mathscr{E}^{3,0}$ of degree $a$ and $\mathscr{E}^{2,1}$ of degree $b$. We then apply the results of Section \ref{sec:impres} to compute the Hodge numbers of the corresponding parabolic cohomology groups in the f\/inal column.

\begin{Remark} Recently there has been a great deal of interest in studying the monodromy subgroups of $\mathrm{Sp}(4,\mathbb{R})$ in the $14$ hypergeometric examples of Doran and Morgan \cite{doranmorgan}. These subgroups may be either arithmetic or non-arithmetic (more commonly called \emph{thin}). Singh and Venkataramana \cite{a4mgacyt, acshg} have proved that the monodromy subgroup is arithmetic in seven of the $14$ cases, and Brav and Thomas \cite{tmsp4} have proved that it is thin in the remaining seven. The arithmetic/thin status of each of the examples is given in the ``Monodromy'' column of Table \ref{tab:msh}.

Computations carried out in \cite{lblefbc} suggest that the arithmetic/thin dichotomy for monodromy groups of hypergeometric local systems is closely related to the value of the \emph{Lyapunov exponents} of the corresponding variations of Hodge structure. Lyapunov exponents are dynamical invariants associated to a f\/lat vector bundle over a hyperbolic Riemann surface, that are computed using the monodromy of the f\/lat bundle as well as the metric structure on the hyperbolic curve.

The main goal of \cite{lblefbc} is to study these Lyapunov exponents; \cite[Lemma 6.3]{lblefbc} (restated above as Theorem \ref{thm:ekmz}) is a step towards this. In \cite{lblefbc} it is conjectured that for the $14$ hypergeometric variations of Hodge structure of Doran and Morgan \cite{doranmorgan}, the associated Lyapunov exponents are closely related to the parabolic degrees of the bundles $\mathscr{E}^{3,0}$ and $\mathscr{E}^{2,1}$, provided that these degrees satisfy a certain relation (see \cite[Conjecture 6.4]{lblefbc} and \cite{fougeron} for the higher rank case).

It is not clear whether there is any relation between the computations presented here and Lyapunov exponents; however, it would be interesting to see whether there is any relation between Lyapunov exponents and the asymptotic growth of $\deg_\mathrm{par} \mathscr{E}^{3,0}$ and $\deg_\mathrm{par} \mathscr{E}^{2,1}$ as the degree $d$ of the covering $[t:s] \mapsto [t^d: s^d]$ grows. Such growth has been investigated in the rank~$2$ case by Kappes~\cite{kappes}.

Finally we note that, according to work of Doran and Malmendier~\cite{cymrsrmt}, most of Doran's and Morgan's~\cite{doranmorgan} $14$ hypergeometric local systems can be obtained from hypergeometric local systems of rank~$3$. These rank $3$ hypergeometric local systems underlie variations of Hodge structure corresponding to families of K3 surfaces, for which the Lyapunov exponents have been computed by Filip \cite{filip}. It seems plausible that the iterated construction of~\cite{cymrsrmt} along with the computations of~\cite{filip} could allow one to compute Lyapunov exponents for the fourteen rank $4$ hypergeometric local systems of~\cite{doranmorgan}.
\end{Remark}

\begin{table}
\centering
\resizebox*{!}{\textheight-6ex}{
\begin{tabular}{|c|c|c|c|c|c|c|}
\hline
\#&$(\alpha_1,\alpha_2,\alpha_3,\alpha_4)$ & $d$ & $a$ & $b$ & Hodge numbers & Monodromy
\\ \hline\hline
1 & $(1/5,\,2/5,\,3/5,\,4/5)$ & $1$ & $0$ & $0$ & $(0,0,0,0,0)$& Thin \\
&& $2$ & $0$ & $0$ & $(0,0,1,0,0)$& \\
&& $5$ & $1$ & $2$ & $(0,0,0,0,0)$& \\
&& $10$ & $2$ & $4$ & $(1,1,1,1,1)$ &\\
\hline
2& $(1/10,\,3/10,\,7/10,\,9/10)$ & $1$ & $0$ & $0$ & $(0,0,0,0,0)$& Arithmetic \\
&& $2$ & $0$ & $0$ & $(0,0,1,0,0)$& \\
& & $5$ & $0$ & \cellcolor{lightgray} $1$ & \cellcolor{lightgray}$(0,1,2,1,0)$&\\
& &$10$ & $1$ & \cellcolor{lightgray} $3$ & \cellcolor{lightgray}$(0,1,3,1,0)$&\\
\hline
3 & $(1/2,\,1/2,\,1/2,\,1/2)$ & $1$ & $0$ & $0$ & $(0,0,0,0,0)$& Thin \\
&& $2$ & $1$ & $1$ & $(0,0,0,0,0)$ &\\
&& $2k$ & $k$ & $k$ & $(k-1,0,0,0,k-1)$ & \\
\hline
4 & $(1/3,\,1/3,\,2/3,\,2/3)$ & $1$ & $0$ & $0$ & $(0,0,0,0,0)$& Arithmetic \\
& & $2$ & $0$ & $0$ & $(0,0,1,0,0)$ &\\
& & $3$ & $1$ & $1$ & $(0,0,0,0,0)$ &\\
& & $6$ & $2$ & $2$ & $(1,0,1,0,1)$ &\\
\hline
5 &$(1/3,\,1/2,\,1/2,\,2/3)$ & $1$ & $0$ & $0$ & $(0,0,0,0,0)$& Thin \\
& & $6$ & $1$ & \cellcolor{lightgray} $2$ & \cellcolor{lightgray}$(0,0,2,0,0)$& \\ \hline
6 & $(1/4,\, 1/2,\, 1/2,\, 3/4)$ & $1$ & $0$ & $0$ & $(0,0,0,0,0)$& Thin \\
& & $4$ & $1$ & $2$ & $(0,0,0,0,0)$& \\
& & $8$ & $2$ & $4$ & $(1,1,0,1,1)$& \\
\hline
7&$(1/8,\,3/8,\,5/8,\,7/8)$ & $1$ & $0$ & $0$ & $(0,0,0,0,0)$& Thin \\
&&$2$ & $0$ & $0$ & $(0,0,1,0,0)$& \\
&& $4$ & $0$ & \cellcolor{lightgray} $1$ & \cellcolor{lightgray}$(0,1,1,1,0)$& \\
&& $8$ & $1$ & \cellcolor{lightgray} $3$ & \cellcolor{lightgray}$(0,1,1,1,0)$& \\ \hline
8 & $(1/6,\, 1/3,\, 2/3,\, 5/6)$ & $1$ & $0$ & $0$ & $(0,0,0,0,0)$& Arithmetic \\
&& $2$ & $0$ & $0$ & $(0,0,1,0,0)$& \\
& & $6$ & $1$ & $2$ & $(0,0,1,0,0)$& \\
\hline
9&$(1/12,\,5/12,\,7/12,\,11/12) $ & $1$ & $0$ & $0$ & $(0,0,0,0,0)$& Thin \\
&& $2$ & $0$ & $0$ & $(0,0,1,0,0)$& \\
&& $3$ & $0$ & \cellcolor{lightgray} $1$ & \cellcolor{lightgray}$(0,1,0,1,0)$& \\
&& $4$ & $0$ & \cellcolor{lightgray} $1$ & \cellcolor{lightgray}$(0,1,1,1,0)$& \\
&& $6$ & $0$ & \cellcolor{lightgray} $2$ & \cellcolor{lightgray}$(0,2,1,2,0)$ &\\
&& $12$ & $1$ & \cellcolor{lightgray}$5$ & \cellcolor{lightgray}$(0,3,1,3,0)$& \\\hline
10 & $(1/4,\, 1/4,\, 3/4,\, 3/4)$& $1$ & $0$ & $0$ & $(0,0,0,0,0)$& Arithmetic \\
&&$2$ & $0$ & $0$ & $(0,0,1,0,0)$& \\
&& $4$ & $1$ & $1$ & $(0,0,1,0,0)$& \\
&& $8$ & $2$ & $2$ & $(1,0,3,0,1)$& \\
\hline
11&$(1/6,\,1/4,\,3/4,\,5/6)$ & $1$ & $0$ & $0$ & $(0,0,0,0,0)$& Arithmetic \\
&& $2$ & $0$ & $0$ & $(0,0,1,0,0)$& \\
&& $12$ & $1$ & \cellcolor{lightgray}$3$ &\cellcolor{lightgray}$(0,1,5,1,0)$& \\ \hline
12&$(1/4,\,1/3,\,2/3,\,3/4)$ & $1$ & $0$ & $0$ & $(0,0,0,0,0)$& Arithmetic \\
&& $2$ & $0$ & $0$ & $(0,0,1,0,0)$& \\
&& $3$ & $0$ & \cellcolor{lightgray}$1$ & \cellcolor{lightgray}$(0,1,0,1,0)$ & \\
&&$12$ & $1$ & \cellcolor{lightgray}$4$ & \cellcolor{lightgray}$(0,2,3,2,0)$& \\ \hline
13&$(1/6,\,1/6,\,5/6,\,5/6)$ & $1$ & $0$ & $0$ & $(0,0,0,0,0)$ & Arithmetic \\
&& $2$ & $0$ & $0$ & $(0,0,1,0,0)$& \\
&& $3$ & $0$ & \cellcolor{lightgray}$0$ & \cellcolor{lightgray}$(0,0,2,0,0)$ &\\
&& $6$ & $1$ & \cellcolor{lightgray} $1$ & \cellcolor{lightgray}$(0,0,3,0,0)$ &\\ \hline
14&$(1/6,\,1/2,\,1/2,\,5/6)$ & $1$ & $0$ & $0$ & $(0,0,0,0,0)$& Thin\\
&& $3$ & $0$ & \cellcolor{lightgray}$1$ & \cellcolor{lightgray}$(0,1,0,1,0)$& \\
&& $6$ & $1$ & \cellcolor{lightgray}$3$ &\cellcolor{lightgray}$(0,1,0,1,0)$&\\
\hline
\end{tabular}}
\caption{Invariants for families of Calabi--Yau threefolds.}
\label{tab:msh}
\end{table}

\begin{Remark}
We pose the question as to whether the above Hodge structures are reducible or not. It is possible that the $h^{2,2}$ value in each of the sets of Hodge numbers in the above table all correspond to rational Hodge classes, thus induce normal functions on the corresponding families of Calabi--Yau threefolds. The maps along which we have pulled back are rigid in the sense that any deformation of this map will change the Hodge numbers, so this indeed seems possible.
\end{Remark}

\subsection{Calabi--Yau fourfolds f\/ibred by mirror quintics}
Finally, we will apply the same approach to classifying Calabi--Yau fourfolds f\/ibred by Calabi--Yau threefolds as we did with Calabi--Yau threefolds f\/ibred by K3 surfaces. We do not obtain concrete classif\/ication results, since modularity results for families of Calabi--Yau varieties of higher dimension are scarce, but we are at least able to place bounds on the space of possibilities.

We start with the mirror quintic family of K3 surfaces. This family can be represented as a~compactif\/ication of the f\/ibres of the Laurent polynomial
\begin{gather*}\frac{(x + y+ z + w + 1)^5}{xyzw}.\end{gather*}
The periods of this family satsify the hypergeometric dif\/ferential equation $L$
\begin{gather*}\delta^4 - t\left(\delta + \frac{1}{5}\right) \left(\delta + \frac{2}{5}\right) \left(\delta + \frac{3}{5}\right) \left(\delta + \frac{4}{5}\right),\end{gather*}
which has local monodromy matrices
\begin{gather*}T_0 = \left( \begin{matrix} 1& 1 & 0 & 0 \\ 0 & 1 & 1 & 0 \\ 0 & 0 & 1& 1 \\ 0 & 0 & 0 & 1 \end{matrix}\right),\qquad T_1 = \left( \begin{matrix} 1& 0 & 0 & 0 \\ 0 & 1 & 0& 0 \\ 0 & 0 & 1& 1 \\ 0 & 0 & 0 & 1 \end{matrix}\right),\qquad T_\infty = \left( \begin{matrix} \zeta_5 & 0 & 0 & 0 \\ 0 & \zeta_{5}^2 & 0 & 0 \\ 0 & 0 & \zeta_{5}^3 & 0 \\ 0 & 0 & 0 & \zeta_{5}^4 \end{matrix}\right),\end{gather*}
and Riemann scheme
\begin{gather*}\left\{
\begin{matrix}
0 & 1 & \infty \\ \hline
0 & 0 & 1/5 \\
0 & 1 & 2/5 \\
0 & 1 & 3/5 \\
0 & 2 & 4/5
\end{matrix} \right\}.\end{gather*}
We can ask: along which maps $g\colon \mathbb{P}^1 \to \mathbb{P}^1$ does the pull-back $g^*L$ have Hodge number $h^{0,4} = 1$? In other words, when can the fourfold total space obtained from the mirror quintic family be a~Calabi--Yau fourfold?

We have all the tools to perform this calculation. Let us take a~map $g\colon \mathbb{P}^1 \to \mathbb{P}^1$ of degree~$d$. Let $k$, $\ell$, and $m$ denote the numbers of points over $0$, $\infty$, and $1$ respectively and let $[y_1,\ldots,y_{\ell}]$ denote the partition of $d$ encoding the ramif\/ication prof\/ile over~$\infty$. Finally, let $\ell_0$ be the number of values of $i$ such that $5 \nmid y_i$. We compute the Hodge numbers of the parabolic cohomology associated to~$g^*L$.

\begin{Proposition}
The degrees of the Hodge bundles of the variation of Hodge structure on $g^*L$ are given by
\begin{gather*}
\deg \mathscr{E}^{3,0} = \frac{1}{2}\bigg( d - \ell_0 - \sum_{i=1}^\ell \big( \lfloor\tfrac{3y_i}{5} \rfloor - \lfloor \tfrac{2y_i}{5} \rfloor \big)\bigg) - \sum_{i=1}^\ell \big( \lfloor\tfrac{2y_i}{5} \rfloor - \lfloor \tfrac{y_i}{5} \rfloor \big)\\
\deg \mathscr{E}^{2,1} = \frac{1}{2}\bigg( d - \ell_0 - \sum_{i=1}^\ell \big( \lfloor\tfrac{3y_i}{5} \rfloor - \lfloor \tfrac{2y_i}{5} \rfloor \big)\bigg),
\end{gather*}
and the equations $\deg \mathscr{E}^{3,0} + \deg \mathscr{E}^{0,3} = -\ell_0$ and $\deg \mathscr{E}^{2,1} + \deg \mathscr{E}^{1,2} = -\ell_0$.
\end{Proposition}

\begin{Remark}
From this data, one can compute all Hodge numbers of the parabolic cohomology groups, by applying the results of Section~\ref{sec:impres}.
\end{Remark}
\begin{proof}
By Theorem \ref{thm:ekmz}, the map $\theta_1\colon \mathscr{E}^{2,1} \rightarrow \mathscr{E}^{1,2} \otimes \Omega_{\mathbb{P}^1}(\Delta)$ has cokernel of length
\begin{gather*}r + \sum_{i=1}^\ell \big( \lfloor\tfrac{3y_i}{5} \rfloor - \lfloor \tfrac{2y_i}{5} \rfloor \big) - (\ell - \ell_0)\end{gather*}
(note that $(\ell - \ell_0) $ counts the number of values of $i$ so that $5 \mid y_i$). Thus
\begin{gather*}
\deg \mathscr{E}^{2,1} - \deg \mathscr{E}^{1,2} = k + \ell + m -2 -r - \sum_{i=1}^\ell \big( \lfloor\tfrac{3y_i}{5} \rfloor - \lfloor \tfrac{2y_i}{5} \rfloor \big).
\end{gather*}
Applying the Riemann--Hurwitz formula, we see that
\begin{gather}\label{eq:RH}
2d - 2 = \sum_{p}(e_p -1) = r + (d-k)+ (d-\ell) + (d-m),
\end{gather}
which implies that $d = k+ \ell + m - r-2$. Therefore,
\begin{gather*}\deg \mathscr{E}^{2,1} - \deg \mathscr{E}^{1,2} = d- \sum_{i=1}^\ell \big( \lfloor\tfrac{3y_i}{5} \rfloor - \lfloor \tfrac{2y_i}{5} \rfloor \big).\end{gather*}

Equation~\eqref{eq:pardeg} and Lemma \ref{lemma:conj} also give that $\deg \mathscr{E}^{3-i,i} + \deg \mathscr{E}^{i,3-i}= -\ell_0$. So we conclude that
\begin{gather*}\deg \mathscr{E}^{2,1} = \frac{1}{2}\bigg( d - \ell_0 - \sum_{i=1}^\ell \big( \lfloor\tfrac{3y_i}{5} \rfloor - \lfloor \tfrac{2y_i}{5} \rfloor \big)\bigg).\end{gather*}

We also have that the cokernel of the map $\theta_0\colon \mathscr{E}^{3,0} \rightarrow \mathscr{E}^{2,1}\otimes \Omega_{\mathbb{P}^1}(\Delta)$ has length
\begin{gather*}r + d + \sum_{i=1}^\ell \big( \lfloor\tfrac{2y_i}{5} \rfloor - \lfloor \tfrac{y_i}{5} \rfloor \big) - (\ell - \ell_0),\end{gather*}
thus
\begin{gather*}
 \deg \mathscr{E}^{3,0} - \deg \mathscr{E}^{2,1} = k+ \ell + m - 2 - r - d - \sum_{i=1}^\ell \big( \lfloor\tfrac{2y_i}{5} \rfloor - \lfloor \tfrac{y_i}{5} \rfloor \big) =- \sum_{i=1}^\ell \big( \lfloor\tfrac{2y_i}{5} \rfloor - \lfloor \tfrac{y_i}{5} \rfloor \big),
\end{gather*}
where the second line follows by applying equation \eqref{eq:RH}. Adding the expression for $\deg \mathscr{E}^{2,1}$ obtained above, we get
\begin{gather*}\deg \mathscr{E}^{3,0} = \frac{1}{2}\bigg( d - \ell_0 - \sum_{i=1}^\ell \big( \lfloor\tfrac{3y_i}{5} \rfloor - \lfloor \tfrac{2y_i}{5} \rfloor \big)\bigg) - \sum_{i=1}^\ell \big( \lfloor\tfrac{2y_i}{5} \rfloor - \lfloor \tfrac{y_i}{5} \rfloor \big),\end{gather*}
which completes the proof of the theorem.
\end{proof}

Since we have that $\deg \mathscr{E}^{3,0} + \deg\mathscr{E}^{0,3} = - \ell_0$, it follows that
\begin{gather*}\deg \mathscr{E}^{0,3} = -\frac{1}{2}\bigg(d + \ell_0 - \sum_{i=1}^\ell \big( \lfloor\tfrac{3y_i}{5} \rfloor - \lfloor \tfrac{2y_i}{5} \rfloor \big)\bigg) + \sum_{i=1}^\ell \big( \lfloor\tfrac{2y_i}{5} \rfloor - \lfloor \tfrac{y_i}{5} \rfloor \big).\end{gather*}
Therefore,

\begin{Corollary}
The Hodge structure on $g^*L$ is Calabi--Yau if and only if
\begin{gather*}\frac{1}{2}\bigg(d + \ell_0 - \sum_{i=1}^\ell \big( \lfloor\tfrac{3y_i}{5} \rfloor - \lfloor \tfrac{2y_i}{5} \rfloor \big)\bigg) - \sum_{i=1}^\ell \big( \lfloor\tfrac{2y_i}{5} \rfloor - \lfloor \tfrac{y_i}{5} \rfloor \big) = 2.\end{gather*}
\end{Corollary}

\begin{Remark}
Note that the value of $h^{0,3}$ of the parabolic cohomology only depends on the structure of ramif\/ication over the point~$\infty$, just as in the case of quartic threefolds.
\end{Remark}

A consequence of this corollary is that there is a f\/inite and enumerable set of values for $[y_1,\dots, y_\ell]$ so that $g^*L$ is Calabi--Yau. This set strictly contains the possible ramif\/ication data over $\infty$ under which the total space of the pullback of the family of mirror quintics is Calabi--Yau. The list of all values of $[y_1,\dots , y_\ell]$ are as follows:
 \begin{gather*} [1, 1], [1,2], [1,3],[1,4],[1,5], [2, 2], [2,3],[2,4],[2,5], \\
 [3,3], [3,4], [3,5], [4,4], [4,5], [5,5], [6], [7], [8], [9], [10]. \end{gather*}
We expect that a proper subset of these ramif\/ication indices should correspond to maps along which we can pull back the family of mirror quintics and get a~variety that admits a smooth (or terminal) Calabi--Yau resolution.

One can also check that the only cases in which the Hodge number $h^{3,0}$ of the parabolic cohomology is 0 are the cases where $\ell = 1$ and $y_1 \in \{1,2,3,4,5\}$.

\begin{Remark}
 Similar calculations can presumably be completed for any family of Calabi--Yau threefolds realising one of the fourteen hypergeometric variations of Hodge structure of type $(1,1,1,1)$ classif\/ied by Doran and Morgan~\cite{doranmorgan}. This includes the so-called ``mirror twin'' families~\cite[Section 3]{doranmorgan}, which have the same Picard--Fuchs equations and $\mathbb{R}$-VHS's as some well-known toric examples, but have dif\/ferent underlying $\mathbb{Z}$-VHS's and Hodge number $h^{1,1}$. In particular, the calculations above apply equally to the ``quintic mirror twin'' family, which shares its Picard--Fuchs equation and $\mathbb{R}$-VHS with the quintic mirror example above.
\end{Remark}

\subsection*{Acknowledgements}

C.F.~Doran (University of Alberta) was supported by the Natural Sciences and Engineering Research Council of Canada, the Pacif\/ic Institute for the Mathematical Sciences, and the Visiting Campobassi Professorship at the University of Maryland.

A.~Harder (University of Miami) was partially supported by the Simons Collaboration Grant in \emph{Homological Mirror Symmetry}.

A.~Thompson (University of Warwick/University of Cambridge) was supported by the Engineering and Physical Sciences Research Council programme grant \emph{Classification, Computation, and Construction: New Methods in Geometry}.

\pdfbookmark[1]{References}{ref}
\LastPageEnding

\end{document}